\documentclass[a4paper, 11pt]{amsart}
\usepackage{enumerate}
\usepackage{mathrsfs}
\usepackage{amssymb,amsmath}
\usepackage[all]{xy}

\title[]{On numerical Newton-Okounkov bodies and the existence of Minkowski bases}
\author{William F. Sawin and David Schmitz}

\thanks{The first author was supported by NSF grant DGE-1148900. The second author was supported by DFG research fellowship SCHM 3223/1-1}


\address{David Schmitz,
	Fach\-be\-reich Ma\-the\-ma\-tik und In\-for\-ma\-tik,
	Philipps-Uni\-ver\-si\-t\"at Mar\-burg,
	Hans-Meer\-wein-Stra{\ss}e,
	D-35032~Mar\-burg, Germany.}
\email{schmitzd@mathematik.uni-marburg.de}

\address{William F. Sawin,
Department of Mathematics,
Princeton university,
Princeton,
NJ 08544}
\email{wsawin@math.princeton.edu}

\newcommand{\R}{\mathbb{R}}
\newcommand{\N}{\mathbb{N}}

\renewcommand{\phi}{\varphi}

\newcommand{\num}{\scriptsize\mbox{num}}

\newcommand{\Eff}{\overline{\mbox{Eff}}}

\newcommand{\vol}{\mbox{vol}}

\renewcommand{\to}{\longrightarrow}

\renewcommand{\epsilon}{\varepsilon}
\newcommand{\Deltanum}{\Delta_{Y_\bullet}^{\num}}

\newtheorem{prop}{Proposition}
\newtheorem{lemma}[prop]{Lemma}

\newtheorem{introthm}{Theorem}

\theoremstyle{definition}
\newtheorem{defin}[prop]{Definition}
\newtheorem{notation}[prop]{Notation}

\newtheorem{remark}[prop]{Remark}

\begin{document}
 
\begin{abstract}
Towards the boundary of the big cone, Newton--Okounkov bodies do not vary continuously and in fact the body of a boundary class is not well defined. Using the global Okounkov body one can nonetheless define a numerical invariant, the numerical Newton--Okounkov body. We show that if a normal projective variety has a rational polyhedral global Okounkov body, it admits a Minkowski basis provided one includes numerical Newton--Okounkov bodies above non-big classes. Under the same assumption, we also show that the dimension of the numerical Okounkov body is the numerical Kodaira dimension. 
\end{abstract}

\maketitle

\section*{Introduction}

In \cite{ssbs} the authors prove that the existence of a Minkowski basis 
for Newton-Okounkov bodies of a normal projective variety $X$ with respect to an
admissible flag $Y_\bullet$ implies rational polyhedrality of the global Okounkov body
$\Delta_{Y_\bullet}(X)$. It remained open whether the converse is also true. In this note, we 
prove that this is indeed the case if we do not require all the the building blocks of Newton--Okounkov bodies to be Newton--Okounkov bodies themselves. Concretely, we prove:
\begin{introthm}\label{ThA}
	Let $X$ be a normal projective variety of dimension $n$. 
	If $\Delta_{Y_\bullet}(X) \subset \R^n\times N^1(X)_\R$ is rational polyhedral, 
	then there exist pseudo-effective 	divisors $D_1,\dots,D_r$ on 
	$X$ together with rational polytopes $\Delta_1,\dots,\Delta_r\subset\R^n$
	such that for any big divisor $D$ there exist non-negative numbers
	$\alpha_1,\dots,\alpha_r$ such that
	\begin{itemize}
		\item $D=\sum\alpha_iD_i$, and
		\item $\Delta_{Y_\bullet}(D) = \sum \alpha_i\Delta_i$.
	\end{itemize}
\end{introthm}

Our method of proof is to first give a chamber decomposition of the pseudoeffective cone such that $\Delta_{Y_{\bullet}}(D)$ varies linearly with $D$ in each chamber. We show that, as long as $D_1,D_2$ lie in the same chamber, we have $\Delta_{Y_{\bullet}}(a D_1 + b D_2) = a \Delta_{Y_{\bullet}}(D_1) + b \Delta_{Y_{\bullet}}(D_2)$. We conclude that the extremal rays of the chambers form the set of divisors in a Minkowski basis for $X$ in the above sense, where for the polytopes $\Delta_i$ we pick the fibers over the classes
$[D_i]$ in the global Okounkov body $\Delta(X)$, so in case $D_i$ is big we have 
$\Delta_i=\Delta_{Y_\bullet}(D_i)$. In any case, following S\'ebastien Boucksom 
(cf. \cite{Bou}) we call the fiber over $[D]$ of a pseudo-effective
divisor $D$ in the global Okounkov body the \emph{numerical Newton--Okounkov body} and denote it by 
$\Deltanum(D)$. 

The second result of this note deals with the numerical Okounkov bodies defined above in the non-big case. We show that in our setting they still encode important information about the corresponding linear series. Since different representatives of a pseudo-effective non-big numerical class can have essentially  different mapping behavior, the usual Kodaira dimension is not well defined for such a class. To remedy this pathology, many notions of numerical dimension have been introduced and studied (see \cite{n04}, \cite{bdpp}, \cite{lehmann}, \cite{eckl}). It turns out that many of those notions agree (see \cite{lehmann}, and \cite{eckl}). In Definition \ref{def:nu}  we follow \cite[Definition 0.3]{eckl}, setting the numerical Kodaira dimension $\nu(D)$ of a pseudo-effective divisor $D$ to be
$$
	\nu(D)=\nu_{\mbox{\scriptsize Vol}}(D)= \max\left\{ k\in \N\mid \exists\  C\in \R:
		\vol_X(D-tA)>t^{\dim X-k}  \right\}
$$
for an ample divisor $A$.

Using this definition we prove the following.
\begin{introthm}\label{ThB}
	Let $D$ be a pseudo-effective divisor on a normal projective variety $X$ and 
	$Y_\bullet$ an admissible flag such that the global Okounkov body 
	$\Delta_{Y_\bullet}(X)$ is locally polyhedral around the fiber over $D$. Then 
	$$
		 \dim\Deltanum(D) = \nu(D).
	$$
\end{introthm}
Boucksom proved that the inequality
$$
	\dim\Deltanum(D) \le \nu(D)
$$
holds in general. In fact, dropping the polyhedrality condition, our proof 
still yields this general inequality, by a slightly more explicit version of Boucksom's proof. 

%

\section{Existence of a  Minkowski Basis}

Let us recall some results from convex geometry. For a polyhedron $\Delta$ and a point $x$ in $\Delta$ the minimal face containing $x$ is just the intersection of all faces of $\Delta$ containing $x$. It has the following intrinsic description:
\begin{lemma}
There is exactly one face $F$ of $\Delta$ containing $x$ in its interior, i.e., such that for each $y$ in $F$, there is some $\epsilon>0$ such that $x - \epsilon (y-x)\in F$. This is the minimal face containing $x$.
\end{lemma}

%

\begin{lemma}\label{minimal} Let $\Delta$ be a polyhedron, $x$ a point in $\Delta$, $F$ the minimal face containing $x$, and $V$ an affine space containing $x$. Then $F \cap V$ is the minimal face of $\Delta \cap V$ containing $x$. \end{lemma}

\begin{proof} The condition of the previous lemma is clearly preserved by restriction to an affine subspace. \end{proof}

\begin{defin} Let $f: \mathbb R^n \to \mathbb R^m$ be a linear map. Let $\Delta$ be a fan in $\mathbb R^n$. Define the projection of $\Delta$ along $f$ to be the set of cones generated by faces and intersections starting from the projection of each cone in $\Delta$ along $f$. \end{defin}

\begin{lemma}\label{projection} If a fan $\Delta$ has finitely many cones, all rational polyhedral, then its projection along any linear map has finitely many cones, all rational polyhedral. \end{lemma}

\begin{proof} Each cone of the largest possible dimension in the projected fan is the intersection of projections of cones from the original fan. There are finitely many cones, hence finitely many projections, and finitely many intersections. Furthermore, all of these cones, being intersections of rational polyhedral cones, are rational polyhedral. Then the cones of the second-largest dimension are intersections of faces of these intersections with images of cones, which must also be finite, because a rational polyhedral cone has finitely many faces, and rational polyhedral, because a face of a rational polyhedral cone is rational polyhedral, and so on. \end{proof}

Let now $X$ be a normal projective variety admitting a flag $Y_\bullet$ such that the global Okounkov body $\Delta_{Y_{\bullet}}(X)$ is rational polyhedral. 

\begin{notation} We denote by $pr_{2* } \Delta_{Y_{\bullet}}(X)$ the projection onto $N^1(X)_\R$ of the fan of all faces of $\Delta_{Y_{\bullet}}(X)$. \end{notation}

\begin{lemma}\label{pair} Let $D_1$ and $D_2$ be two divisors that lie in a cone $C$ of $pr_{2*}\Delta_{Y_{\bullet}}(X)$. If $C$ is the minimal cone containing $aD_1+bD_2$, then $\Deltanum(a D_1 + b D_2) = a \Deltanum(D_1) + b \Deltanum(D_2)$. \end{lemma}

\begin{proof} 
Let $x$ be an extremal point of $\Deltanum( a D_1 + b D_2) = \Delta_{Y_\bullet}( X) \cap pr_2^* (aD_1+bD_2)$. Let $F$ be the minimal face of $\Delta_{Y_{\bullet}}(X)$ containing $x$. Then the image of $F$ under $pr_2$ intersects $C$, and from the minimality of $C$ it follows that the image of $F$ under $pr_2$ contains $C$. Hence there are also points  $y_1$ in $ \Delta_{Y_\bullet}( X) \cap pr_2^* (D_1)$ and $y_2$ in $ \Delta_{Y_\bullet}( X) \cap pr_2^* (D_2)$ lying in $F$. Then $a y_1 + by_2$ lies in $F \cap pr_2^* (a D_1 + b D_2)$. Because $F$ is the minimal face containing $x$, by Lemma \ref{minimal}, $F \cap pr_2^* (a D_1 + b D_2)$ is the minimal face of $ \Delta_{Y_\bullet}( X) \cap pr_2^* (aD_1+bD_2)$ containing $x$. Because $x$ is an extremal point, this implies that $F \cap pr_2^* (a D_1 + b D_2) = \{ x\}$, so  $ay_1 + b y_2 = x$. 
 
Hence all the extremal points of $\Deltanum(a D_1 + b D_2)$ lie in $a \Deltanum(D_1) + b \Deltanum(D_2)$, which implies that the whole convex body lies in $a \Deltanum(D_1) + b \Deltanum(D_2)$. The converse is trivial.
\end{proof}


\begin{proof}[Proof of Theorem \ref{ThA}] We take $D_i$ to be the generators of the one-dimensional cones of $pr_{2*} \Delta_{Y_\bullet}(X)$ and take $\Delta_i$ to be $\Delta_{Y_{\bullet}}(D_i)$. By Lemma \ref{projection} there indeed a finitely many $D_i$, and the $\Delta_i$ are indeed rational polytopes.

To finish the result it is sufficient to prove that, for each divisor $D$, we can write $D$ as a linear combination of the $D_i$, such that the numerical Newton--Okounkov body of $D$ is the weighted Minkowski sum of the corresponding numerical Okounkov bodies $\Delta_i$. Then for a big divisor $D$, we will obtain the same identity for the usual Newton-Okounkov body of $D$, because it is equal to the numerical one.

We prove this by induction on the dimension of the minimal cone of $pr_{2*} \Delta_{Y_\bullet}(X)$ containing $D$.  It is clearly true if $D$ is contained in a one-dimensional cone. For arbitrary $D$, take a line through $D$, and let $D_a$ and $D_b$ be the points where it intersects the boundary of the minimal cone containing $D$. We may write $D$ as a linear combination $a D_a + bD_b$, and by Lemma \ref{pair}, $\Deltanum(a D_a + b D_b) = a \Deltanum(D_a) + b \Deltanum(D_b)$. Because $D_a$ and $D_b$ lie in the boundary of the minimal cone containing $D$, they lie in some face of the minimal cone containing $D$, and hence in a lower-dimensional cone. So by induction we may write them as a linear combination of $D_i$ and write their numerical Newton-Okounkov bodies as a linear combination of the $\Delta_i$. Thus we can do the same for $D$, verifying the induction step. \end{proof}

\section{Numerical Newton--Okounkov bodies}

The numerical Okounkov bodies defined above are just ordinary Okounkov bodies 
in the case of big divisors. However, on the boundary of the pseudo-effective cone,
the Newton--Okounkov body is not a numerical invariant. Even the Kodaira dimension is not invariant.
Nonetheless, notions of numerical Kodaira dimension have been studied that 
are numerical invariants, even on the boundary. We refer the reader to the discussion in 
\cite{lehmann} and \cite{eckl}. We follow the definition from \cite{eckl}.
\begin{defin}\label{def:nu}
	Let $D$ be a pseudo-effective divisor and let $A$ be some fixed ample divisor.
	The numerical Kodaira dimension of $D$ is
	$$
		\nu(D):=\max\left\{ k\in \N\mid \exists\  C\in \R \mbox{ such that } 
		\vol_X(D-tA)>t^{\dim X-k}  \right\}.
	$$
\end{defin}
Since the volume of $D-tA$ only depends on the numerical class
 for each $t$, the numerical Kodaira dimension $\nu(D)$ is a numerical invariant.
 Let us prove Theorem \ref{ThB} stating that $\nu(D)$ agrees with the dimension 
 of the numerical Newton--Okounkov body $\Delta_{Y_\bullet}^{\num}(D)$ for any 
 admissible flag $Y_\bullet$ for which the global Okounkov body is polyhedral around 
 the fiber of $D$. 
 
 We use the following property of projections of cones.
 \begin{lemma}\label{l:bounded}
 	Let $f:\R^{N}\to \R^n$ be a linear map and let  $\mathcal C\subset \R^N$ be a 
 	closed convex cone. Furthermore, let $R\subset f(\mathcal C)$ be a 
 	subcone and assume that around $f^{-1}(R)$ the cone $\mathcal C$ is polyhedral. Denote
 	by $d_N, d_n$ Euclidean distance in $\R^N, \R^n$, respectively.
 	Then the function
 	\begin{eqnarray*}
 		\rho: \mathcal C\setminus f^{-1}(R) &\to& \R\\
 		x &\longmapsto& \frac{d_N(x,f^{-1}(R)}{d_n(f(x),R)}
 	\end{eqnarray*}
 	 is bounded.
 \end{lemma}
 \begin{proof}
 	The preimage $f^{-1}(R)$ is again a cone in $\R^N$. Thus for any $\lambda\ge 0$, 
 	$x\in\mathcal C$ we have 
 	$$
 		d_N(\lambda x, f^{-1}(R)) = d_N(\lambda x, \lambda f^{-1}(R)) =
 		\lambda\cdot d_N(x, f^{-1}(R)),
 	$$
	By the same token and linearity of $f$,
 	$$
 		d_n( f(\lambda x), R) = 	\lambda\cdot d_n(f(x), R),
 	$$
 	so the function $\rho$ is invariant under scaling with positive reals. 
 	We can thus assume it to be a function on 
 	the quotient of $ (\mathcal C\setminus f^{-1}(R))$ by the scaling action of the multiplicative group of positive reals, which we call $Y$. 
	
 	Note that for any $x\in f^{-1}(R)$, $y\in \R^N$, $0<t\ll1$ we have 
 	\begin{eqnarray*}
 		\rho(x+ty) &=& \frac{d_N(x+ty,f^{-1}(R))}{d_n(f(x+ty),R)}\\
 				&=& \frac{d_N(ty,f^{-1}(R))}{d_n(f(x)+tf(y)),R)}\\
 				&=& \frac{td_N(y,f^{-1}(R))}{td_n(f(y),R)}=\rho(y).\\
 	\end{eqnarray*}
 	Thus $\rho$ is constant along each ray from $x\in f^{-1}(R)$. Let 
 	$U\subset Y$ be
 	a scale-invarinat open neighborhood of $f^{-1}(R)$. Then $\rho$ is continuous on the compact set 
 	$Y\setminus U$, and thus is bounded on $Y \setminus U$. On the other hand choosing $U$ 
	small enough, all points in $U\setminus f^{-1}R$ lie in some ray $\{x-ty\}$ for
 	 $x\in f^{-1}(R)$ and $y\in Y \setminus U$ by the polyhedrality of $\mathcal C$ around 
 	 $f^{-1}(R)$, 
 	 thus their values under $\rho$ are 
 	 bounded by the maximum on $Y\setminus U$.
 \end{proof}

\begin{proof}[Proof of Theorem \ref{ThB}]
	Let $Y_\bullet$ be an admissible flag on $X$ and $\Delta(X)$ the corresponding 
	global Okounkov body in $\R^N=N^1(X)_\R\times\R^d$ and let $D$ be a 
	pseudo-effective divisor and $A$ an ample divisor on $X$.  
	Let us denote by $\Delta_\epsilon$ the full-dimensional standard simplex in 
	$\R^N$ of side-length $\epsilon$.
	It suffices to prove the existence of two positive constants $C_1,C_2$ such that for 
	all $0<t\ll1$ we have
	$$
		\Deltanum(D)+\Delta_{t\cdot C_1}\subseteq \Delta_{Y_\bullet}(D+tA)
		\subseteq \Deltanum(D)+\Delta_{t\cdot  C_2},
	$$
	since this clearly implies the existence of positive constants $c_1,c_2$ with
	$$
		c_1t^{n-r}<\vol(D+tA)<c_2 t^{n-(r+1)}
	$$
	for all $0<t\ll1$, where $r$ denotes the dimension of $\Deltanum(D)$. (It is even sufficient to show these inclusions up to translation, which is what we will do.)
	For the second inclusion take in Lemma \ref{l:bounded} for $f$ the projection  
	$\R^N\to N^1(X)_\R$, $\mathcal C=\Delta(X)$, and $R$ the ray in $\Eff(X)$ spanned by
	$D$. Then the boundedness of $\rho$ implies that there is a constant $C$ such that 
	for small $t$ every point of $\Delta(D+tA)$ has distance at most $Ct$ from 
	a point in $\Deltanum(D)$. We can choose $C_2$ large enough that $\Delta_{t \cdot C_2}$ contains a sphere of radius $Ct$ so that $\Deltanum(D)+\Delta_{t\cdot  C_2}$ contains $\Delta_{Y_\bullet}(D+tA)$.
	
	For the first inclusion, the convexity of the cone $\Delta(X)$ yields 
	the inclusion of numerical Okounkov bodies (fibers under the projection to $N^1(X)_\R$)
	$$
		\Deltanum(D)+\Deltanum(tA) \subseteq \Deltanum(D+tA),
	$$ 
	so by the bigness of $A$ and $D+tA$, we get
	$$
		\Deltanum(D)+\Delta_{Y_\bullet}(tA) \subseteq \Delta_{Y_\bullet}(D+tA).
	$$
	Recall that for any ample divisor $A$ the Newton--Okounkov body 
	$\Delta_{Y_\bullet}(A)$ contains a simplex $\Delta_\epsilon$ (see \cite[Theorem B]{kl15}).
	Therefore, the left hand side above contains the Minkowski sum
	$
		\Deltanum(D)+\Delta_{t\cdot\epsilon},
	$
	which concludes the proof.
\end{proof}
	
\begin{remark}
	The above proof shows that the inequality
	$$
		\dim\Deltanum(D) \le \nu(D)
	$$
	holds regardless of the shape of $\Delta(X)$. This has already been observed by 
	Boucksom in \cite[Lemme 4.8]{Bou}.  
	In particular, if $D$ is not big and its numerical Newton--Okounkov body has dimension 
	dim$(X)-1$, then the numerical Kodaira dimension must agree. 
\end{remark}



\end{document}